\documentclass[a4paper, final, notitlepage, 10pt]{article}
\usepackage{latexsym,bm}
\usepackage{amsthm}
\usepackage{amsfonts}
\usepackage{amsmath}
\usepackage[all]{xy}
\usepackage[final]{graphics}
\usepackage{fancyhdr}
\usepackage{verbatim}
\usepackage{geometry}
\usepackage{xcolor}
\usepackage{enumerate}
\usepackage{showkeys}
\usepackage{amssymb}
\usepackage{adjustbox}
\usepackage[flushleft]{threeparttable}

\newtheorem{thm}{Theorem}[section]
\newtheorem{cor}[thm]{Corollary}
\newtheorem{lem}[thm]{Lemma}
\newtheorem{prop}[thm]{Proposition}
\theoremstyle{definition}

\newtheorem{rem}[thm]{Remark}

\newtheorem{ques*}[thm]{Question}
\newtheorem*{theorem*}{Question}

\numberwithin{equation}{section}

\newcommand{\thmref}[1]{Theorem~\ref{#1}}

\newcommand{\lemref}[1]{Lemma~\ref{#1}}
\newcommand{\propref}[1]{Proposition~\ref{#1}}

\newcommand{\ZZ}{\mathbb{Z}}

\newcommand{\QQ}{\mathbb{Q}}

\renewcommand{\O}{\mathcal{O}}

\newcommand{\PP}{\mathbb{P}}
\newcommand{\e}{\epsilon}

\newcommand{\mL}{\mathcal{L}}
\newcommand{\tpi}{\tilde{\pi}}

\newcommand{\red}{\textcolor{red}}
\newcommand{\blue}{\textcolor{blue}}

\title{Smooth minimal surfaces of general type with \\
 $p_g=0$, $K^2=7$ and involutions}
\author{Yifan Chen, YongJoo Shin and Han Zhang}
\begin{document}

\date{}
\maketitle

\begin{abstract}
Lee and the second named author studied involutions on smooth minimal surfaces $S$ of general type with $p_g(S)=0$ and $K_S^2=7$. They gave the possibilities of the birational models $W$ of the quotients and the branch divisors $B_0$ induced by involutions $\sigma$ on the surfaces $S$.

In this paper we improve and refine the results of Lee and the second named author. We exclude the case of the Kodaira dimension $\kappa(W)=1$ when the number $k$ of isolated fixed points of an involution $\sigma$ on $S$ is nine. The possibilities of branch divisors $B_0$ are reduced for the case $k=9$, and are newly given for the case $k=11$. Moreover, we show that if the branch divisor $B_0$ has three irreducible components, then $S$ is an Inoue surface.
\end{abstract}

\section{Introduction}
In this article we focus on smooth minimal surfaces of general type with $p_g=0$ and $K^2=7$.
A complete classification of such surfaces is still out of reach. The first examples of such surfaces are
the Inoue surfaces (\cite{inoue}). Any Inoue surface is a Galois cover of the $4$-nodal cubic surface with Galois group, the Klein group $\ZZ/2\ZZ\times \ZZ/2\ZZ$, and its bicanonical map has degree $2$. Moreover, the involution associated the bicanonical map is contained in the Galois group. See \cite[Example~4.1]{bicanonical1}.

Mendes Lopes and Pardini show that the degree of the bicanonical map $\varphi$ is at most $2$ and the case $\deg \varphi=2$ is investigated in details (see \cite[Theorem~1.1]{bicanonical2}). In the case $\deg \varphi=2$, the involution associated to the bicanonical map is called the bicanonical involution. On the the hand, Lee and the second named author study involutions on $S$ in \cite{leeshin}.

Before proceeding, we introduce several notations. For a smooth minimal surface $S$ of general type with $p_g(S)=0$, $K_S^2=7$ and an involution $\sigma$, we denote by $R$ the divisorial fixed part of $\sigma$, by $k$ the number of isolated fixed points of $\sigma$, by
$\pi \colon S \rightarrow \Sigma$ the quotient map, and by $\eta \colon W \rightarrow \Sigma$ the minimal resolution of $\Sigma$. The map $\epsilon$ is the blow-up of $S$ at $k$ isolated fixed points of $\sigma$, and $\tpi$ is induced by the quotient map $\pi$. Then we have the commutative diagram:
\begin{align*}
\xymatrix{
V \ar"1,2"^{\epsilon} \ar"2,1"_{\tpi} & S \ar"2,2"^{\pi}\\
W \ar"2,2"^{\eta}                    &  \Sigma.
}
\end{align*}
Moreover, set $B:=\pi(R)$ and $B_0:=\eta^*(B)$. Since the surface $W$ has $k$ nodal curves $N_i,\ i=0,1,\ldots,k-1$ induced by the isolated fixed points of $\sigma$, the map $\tilde{\pi}$ is a flat double cover branched on $B_0+\sum_{i=0}^{k-1} N_i$. Then there exists a divisor $\mL$ such that
\begin{align}
2\mL\equiv B_0+\sum_{i=0}^{k-1} N_i \textrm{ and } \tilde{\pi}_*\mathcal{O}_V=\mathcal{O}_W\oplus \mathcal{O}_W (-\mL), \label{eq:coveringdata}
\end{align}
 where $\equiv$ is the linear equivalence relation.
For an irreducible component $\Gamma$ of $B_0$ with $g(\Gamma)=a$ and $\Gamma^2=b$, we denote it by $\substack{\Gamma\\(a,b)}$.

Lee and the second named author describe all the possible cases for $W$ and $B_0$
\cite[Table in page 3]{leeshin}, except the case where $\sigma$ is the bicanonical involution (equivalently $k=11$).
Looking into the case $k=9$ in \cite{leeshin}, we see that $\kappa(W)\le 1$ and the list of possibilities of the branch divisor $B_0$ is long except when $\kappa(W)=0$. From the result of \cite{bicanonical2}, if $\sigma$ is the bicanonical involution, then $W$ is a rational surface with $K_W^2=-4$ and $11$ pairwise disjoint nodal curves.
However, the branch divisor $B_0$ is not described in \cite{bicanonical2} either.\\

Section~2 of this article deals with the case $k=9$ by excluding the case $\kappa(W)=1$ and making the list in \cite{leeshin} of possibilities of $B_0$ much shorter.
\begin{thm}\label{thm:case9main}
Let $S$ be a smooth minimal surface of general type with $p_g(S)=0$, $K_S^2=7$ and an involution $\sigma$. Let $k$ be the number of isolated fixed points of $\sigma$.

Assume that $k=9$. Then $K_W^2=-2$ and either $W$ is birational to an Enriques surface or $W$ is a rational surface.

Moreover, if $W$ is a rational surface,
$B_0$ has the following possibilities:
   \begin{enumerate}[\upshape (1)]
    \item $B_0=\substack{\Gamma_0 \\ (2, 0)}+\substack{\Gamma_1 \\ (2, 0)}+\substack{\Gamma_2 \\ (1,-2)}$;
    \item $B_0=\substack{\Gamma_0 \\ (3, 0)}+\substack{\Gamma_1 \\ (1, -2)}$;
    \item $B_0=\substack{\Gamma_0 \\ (2, -2)}+\substack{\Gamma_1 \\ (2, 0)}$;
    \item $B_0=\substack{\Gamma_0 \\ (3,-2)}$.
    \end{enumerate}
\end{thm}
\noindent
Actually, we prove more. In particular, when $W$ is a rational surface, we exhibit a rational fibration $f \colon W\rightarrow \PP^1$ explicitly, classify the singular fibres of $f$ and calculate the divisor classes of the irreducible components of $B_0$ in terms of $K_W, F$ and certain irreducible components of the singular fibres of $f$, where $F$ is a general fibre of $f$.
See Propositions \ref{prop:rationalfib} and \ref{prop:branchdiv}.\\

We study the case $k=11$ in Section~3 and obtain similar results.
\begin{thm}\label{thm:case11main}
Let $S$ be a smooth minimal surface of general type with $p_g(S)=0$, $K_S^2=7$ and an involution $\sigma$. Let $k$ be the number of isolated fixed points of $\sigma$.

Assume that $k=11$. Then $K_W^2=-4$, $W$ is a rational surface and $B_0$ has the following possibilities:
      \begin{enumerate}[\upshape (1)]
    \item $B_0=\substack{\Gamma_0 \\ (3, 0)}+\substack{\Gamma_1 \\ (2, -2)}$;
    \item $B_0=\substack{\Gamma_0 \\ (3, -2)}+\substack{\Gamma_1 \\ (2, 0)}$;
    \item $B_0=\substack{\Gamma_0 \\ (4,-2)}$.
\end{enumerate}
\end{thm}
\noindent
See Propositions \ref{prop:11rationalfib} and \ref{prop:11branchdiv} for more details.\\

Combining with the result in \cite{leeshin}, we have the following theorem.
Note that (1)-(5) in the following theorem are results of \cite{leeshin} and the assertion $K_W^2=-4$ in (7) is also known in
\cite{bicanonical2}.

\begin{thm}\label{thm:main}
Let $S$ be a smooth minimal surface of general type with $p_g(S)=0$, $K_S^2=7$ and an involution $\sigma$. Let $k$ be the number of isolated fixed points of $\sigma$. Then one of the following cases holds:
\begin{enumerate}[\upshape (1)]
\item $k=5$, $W$ is a smooth minimal surface of general type with $K_W^2=2$ and $B_0=\substack{\Gamma_0 \\ (1,-2)}$.
\item $k=7$, $W$ is a smooth minimal surface of general type with $K_W^2=1$ and $B_0=\substack{\Gamma_0 \\ (3,2)}$.
\item $k=7$, $W$ is a smooth surface of general type with $K_W^2=0$ whose smooth minimal model $W'$ has $K_{W'}^2=1$, and $B_0$ has the following possibilities:
        \begin{enumerate}[\upshape (a)]
            \item $B_0=\substack{\Gamma_0 \\ (2, -2)}$;
            \item $B_0=\substack{\Gamma_0\\ (2, 0)}+\substack{\Gamma_1\\ (1,-2)}$.
        \end{enumerate}
\item $k=7$, $W$ is a smooth minimal properly elliptic surface and $B_0$ has the following possibilities:
        \begin{enumerate}[\upshape (a)]
            \item $B_0=\substack{\Gamma_0 \\ (2, -2)}$;
            \item $B_0=\substack{\Gamma_0\\ (2, 0)}+\substack{\Gamma_1\\ (1,-2)}$.
        \end{enumerate}

\item $k=9$, $W$ is birational to an Enriques surface, $K_W^2=-2$ and $B_0$ has the following possibilities:
        \begin{enumerate}[\upshape (a)]
            \item $B_0=\substack{\Gamma_0 \\ (3,-2)}$;
            \item $B_0=\substack{\Gamma_0 \\ (3,0)}+\substack{\Gamma_1\\ (1,-2)}$.
        \end{enumerate}
\item $k=9$, $W$ is a rational surface with $K_W^2=-2$ and $B_0$ has the following possibilities:
   \begin{enumerate}[\upshape (a)]
    \item $B_0=\substack{\Gamma_0 \\ (2, 0)}+\substack{\Gamma_1 \\ (2, 0)}+\substack{\Gamma_2 \\ (1,-2)}$;
    \item $B_0=\substack{\Gamma_0 \\ (3, 0)}+\substack{\Gamma_1 \\ (1, -2)}$;
    \item $B_0=\substack{\Gamma_0 \\ (2, -2)}+\substack{\Gamma_1 \\ (2, 0)}$;
    \item $B_0=\substack{\Gamma_0 \\ (3,-2)}$.
    \end{enumerate}
\item $k=11$, $W$ is a rational surface with $K_W^2=-4$ and $B_0$ has the following possibilities:
       \begin{enumerate}[\upshape (a)]
    \item $B_0=\substack{\Gamma_0 \\ (3, 0)}+\substack{\Gamma_1 \\ (2, -2)}$;
    \item $B_0=\substack{\Gamma_0 \\ (3, -2)}+\substack{\Gamma_1 \\ (2, 0)}$;
    \item $B_0=\substack{\Gamma_0 \\ (4,-2)}$.
\end{enumerate}
\end{enumerate}
\end{thm}

To the best knowledge of the authors, the following are the known examples of surfaces of general type with $p_g=0$ and $K_S^2=7$:
the $4$-dimensional family of Inoue surfaces \cite{inoue, inouemfd, bicanonical1}, the $3$-dimensional family of examples constructed by the first named author \cite{chennew, commuting}, the $2$-dimensional family of examples constructed by the first named and second named authors \cite{chenshin18} (including the surface constructed by Rito \cite{rito15}), and the surfaces constructed by Calabri and Stagnaro \cite{calabristagnaro}. From all these examples, we see that there are examples for cases (1), (4)(a), (5)(a), (5)(b), (6)(a), (6)(b) and (7)(a) in \thmref{thm:main}. See Section~4 for more details.\\

In \cite[Theorem~1.1]{chenshin18}, the case (7)(a) is completely classified:
if $B_0=\substack{\Gamma_0 \\ (3, 0)}+\substack{\Gamma_1 \\ (2, -2)}$
then $S$ is an Inoue surface.
It follows that  $S$ admits another two more involutions besides $\sigma$.
We prove the following similar theorem for the case (6)(a).
\begin{thm}\label{thm:threecpnts}Let $S$ be a smooth minimal surface of general type with $p_g(S)=0$, $K_S^2=7$ and an involution $\sigma$. Assume that $B_0=\substack{\Gamma_0 \\ (2, 0)}+\substack{\Gamma_1 \\ (2, 0)}+\substack{\Gamma_2 \\ (1,-2)}$. Then $S$ is an Inoue surface.
\end{thm}
See the end of Section~4 for the proof.

\section{The case $k=9$}
In this section, we consider the case $k=9$. We use the same notation in the introduction.
\subsection{Known results}
The starting point is the following theorem by Lee and the second named author.

Denote $D:=2K_W+B_0$. Since $\tilde{\pi}^*D\equiv \epsilon^*(2K_S)$ $$D^2=14, \textrm{ and } D \textrm{ is nef and big.}$$

\begin{thm}[{\cite[Theorem in Page 122 and the list in Page 123 and Theorem 3.5~(iii)]{leeshin}}]\label{thm:known}Assume that $k=9$. Then
$$DK_W=2,~~K_W^2=-2,~~K_WB_0=6,~~B_0^2=-2~\text{and}~\kappa(W)\le 1.$$
\begin{enumerate}[\upshape (1)]
\item Assume further that $\kappa(W)=0$. Then either $B_0=\substack{\Gamma_0 \\ (3,-2)}$ or $B_0=\substack{\Gamma_0 \\ (3,0)}+\substack{\Gamma_1\\ (1,-2)}$.
\item Assume further that $\kappa(W)=-\infty$ or $1$. Then $B_0$ is one of the following:
\begin{enumerate}[\upshape (a)]
    \item $B_0=\substack{\Gamma_0 \\ (4, 2)}+\substack{\Gamma_1 \\ (0, -4)}$.
    \item $B_0=\substack{\Gamma_0 \\ (3,-2)}$.
    \item $B_0=\substack{\Gamma_0 \\ (4, 4)}+\substack{\Gamma_1 \\ (1, -2)}+\substack{\Gamma_2 \\ (0,-4)}$.
    \item $B_0=\substack{\Gamma_0 \\ (4, 4)}+\substack{\Gamma_1 \\ (0, -6)}$.
    \item $B_0=\substack{\Gamma_0 \\ (3, 0)}+\substack{\Gamma_1 \\ (1, -2)}$.
    \item $B_0=\substack{\Gamma_0 \\ (3, 2)}+\substack{\Gamma_1 \\ (1, -4)}$.
    \item $B_0=\substack{\Gamma_0 \\ (2, -2)}+\substack{\Gamma_1 \\ (2, 0)}$.
    \item $B_0=\substack{\Gamma_0 \\ (3, 2)}+\substack{\Gamma_1 \\ (1, -2)}+\substack{\Gamma_2 \\ (1,-2)}$.
    \item $B_0=\substack{\Gamma_0 \\ (2, 0)}+\substack{\Gamma_1 \\ (2, 0)}+\substack{\Gamma_2 \\ (1,-2)}$.
\end{enumerate}
\end{enumerate}

\end{thm}

Before proceeding, we calculate some intersection numbers.
Set $M_j:=jK_W+D$ for $j \ge 1$. Then
\begin{align}
K_WM_j=-2j+2,~~DM_j=2j+14,~~M_j^2=-2j^2+4j+14.  \label{eq:Mj}
\end{align}

\subsection{Exclusion of the case $\kappa(W)=1$}

\begin{lem}\label{lem:-1curve1}Let $C$ be a $(-1)$-curve on $W$. Then
\begin{enumerate}[\upshape(1)]
\item $C$ intersects at least one of the nodal curves $N_0, N_1, \ldots, N_8$\blue{;}
\item $DC\ge 1$.
\end{enumerate}
\end{lem}
\proof
If $C$ is disjoint from $N_0+N_1+\cdots+N_8$,  then we obtain a contradiction to \cite[Lemmas~3.3 and 3.4]{leeshin} by blowing down $C$. Thus $C$ intersects at least one $N_i$, say $N_0$.

Suppose that $DC\le 0$. Since $D$ is nef, $DC=0$ and thus $B_0C=(D-2K_W)C=2$.
Then $D(C+N_0)=0$ and the algebraic index theorem imply
$$2CN_0-3=(C+N_0)^2<0$$ and so $CN_0=1$.
Since $C(B_0+N_0+N_1+\cdots+N_8)=2C\mL$ is even and $CB_0=2$, $C$ intersects another $N_i$, say $CN_1>0$.
The similar argument as above shows that $CN_1=1$.

Then $D(N_0+2C+N_1)=0$ and $(N_0+2C+N_1)^2=0$. This contradicts the algebraic index theorem.\qed

Recall that $M_1=K_W+D$.
\begin{cor} The divisor $M_1$ is nef and big.
\end{cor}
\begin{proof}
By \eqref{eq:Mj}
$$K_WM_1=0,~~DM_1=16,~~M_1^2=16.$$
The Riemann-Roch theorem yields that $h^0(W,\mathcal{O}_W (M_1))=9$.

Assume that $M_1C<0$ for an irreducible curve $C$.
Then $C^2<0$ and
$$K_WC=(M_1-D)C<-DC\le 0.$$
It follows that $C$ is a $(-1)$-curve with $DC=0$, a contradiction to \lemref{lem:-1curve1}.
\end{proof}

\begin{cor}Either $W$ is a rational surface or $W$ is birational to an Enriques surface.
\end{cor}
\proof By \thmref{thm:known}, it suffices to prove that $\kappa(W) \not =1$.

Assume $\kappa(W)=1$.
Let $t \colon W \rightarrow W'$ be the birational morphism to the minimal model $W'$ of $W$.
Then $K_W=t^*K_{W'}+E$ with $E>0$ since $K_W^2=-2$ and $K_{W'}^2=0$.

Observe that $M_1K_W=0$ and thus $M_1t^*K_{W'}+M_1E=0$. Since both $M_1$ and $K_{W'}$ are nef,
we have $M_1t^*K_{W'}=0$. This is impossible since $M_1$ is big and $\kappa(W')=1$.\qed

\subsection{The case $\kappa(W)=-\infty$:~a rational fibration}

In this subsection, we assume further $\kappa(W)=-\infty$ and so $W$ is a smooth rational surface.
We remark that the linear equivalence and numerical equivalence are the same on $W$.

\begin{lem}\label{lem:-1curve2}If $C$ is a $(-1)$-curve, then $DC \ge 3$. Moreover, if $DC=3$, then $B_0C=5$ and $C$ intersects exactly one $N_i$ of the nodal curves $N_0, \ldots, N_8$ with $CN_i=1$.
\end{lem}
\proof

By  \lemref{lem:-1curve1}, $C$ intersects some $N_i$, say $N_0$, and $DC\ge 1$.
Assume that $DC\le 3$. We shall show $DC=3$.

Since $D(C+N_0)\le 3$,
$$2CN_0-3=(C+N_0)^2 \le \frac{(D(C+N_0))^2}{D^2}\le \frac{9}{14}.$$
Then $CN_0=1$.

If $DC=1$, then $D(-K_W+2C+N_0)=0$ and $(-K_W+2C+N_0)^2=0$. By the algebraic index theorem, $K_W \equiv 2C+N_0$, which is impossible since $p_g(W)=0$.

If $DC=2$, then $B_0C=(D-2K_W)C=4$. By \eqref{eq:coveringdata}
$$2C\mL=C(B_0+N_0+\cdots+N_8)=5+C(N_1+\cdots+N_8).$$
Therefore $C$ intersects another $N_i$ say $N_1$. Then $CN_1=1$ as above. Thus
$$(2K_W+D)(N_0+2C+N_1)=0, (2K_W+D)^2=14, (N_0+2C+N_1)^2=0,$$
 which gives a contradiction to the algebraic index theorem.

Therefore we have $DC=3$.
Assume that $C$ intersects another $N_1$. Then $CN_1=1$ as above.
Then $(3K_W+D)^2=8$, $(3K_W+D)(N_0+2C+N_1)=0$ and $(N_0+2C+N_1)^2=0$, a contradiction to the algebraic index theorem.

\qed

\begin{cor}$M_2=2K_W+D$ and $M_3=3K_W+D$ are nef and big.
\end{cor}

\begin{proof}We have seen that $M_1$ is nef and big. Note that $M_2=K_W+M_1$ and by
\eqref{eq:Mj}, $M_2^2=14$ and $K_WM_2=-2$.
Then the Kawamata-Viehweg vanishing theorem and the Riemann-Roch theorem show that
$$h^0(W,\mathcal{O}_W (M_2))=9.$$
Assume that $M_2C<0$ for an irreducible curve $C$. Then $C^2<0$ and $K_WC<-M_1C\le 0$.
Thus $C$ is a $(-1)$-curve with $M_1C=0$, and so $DC=(M_1-K_W)C=1$, a contradiction to the previous lemma.
Therefore $M_2$ is nef and big.

Similar argument shows that $M_3=K_W+M_2$ is nef and big.
\end{proof}

\begin{prop}\label{prop:rationalfib}After possibly renumbering the nodal curves $N_0, \ldots, N_8$, we have
\begin{align}
|4K_W+D|=2|F|+2G+N_0, \label{eq:M4}
\end{align}
where $F$ is a general fibre of a rational fibration $f \colon W \rightarrow \PP^1$ and
\begin{enumerate}[\upshape (1)]
\item $DF=8, B_0F=12$;
\item $G$ is a $(-1)$-curve with $GN_0=1, DG=3, B_0G=5$;
\item $G$ is disjoint from the nodal curves $N_1, \ldots, N_8$;
\item The fibration $f$ contains the following $4$ singular fibres:
         \begin{align}
         N_{2j-1}+2G_j+N_{2j}, j=1, 2, 3, 4, \label{eq:fibres}
         \end{align}
         where for $j=1, 2, 3, 4$, $G_j$ is a reduced curve with $G_j^2=-1$, $G_jN_{2j-1}=G_{j}N_{2j}=1$ and $G_jB_0=6$, and for $j=2, 3, 4$,
         $G_j$ is a $(-1)$-curve.
\item There are two possibilities for the reduced curve $G_1$:
    \begin{enumerate}
     \item[(i)]$G_1=G+N_0+Z$, where $Z$ is a $(-2)$-curve with $ZG=0$, $ZN_0=ZN_1=ZN_2=1$ and $B_0Z=1$. In this case,
      \eqref{eq:fibres} are exactly all the singular fibres of $f$.

      \item[(ii)]  $G_1$ is a $(-1)$-curve. In this case, besides \eqref{eq:fibres}, $f$ contains another singular fibre
      $G+N_0+Z$, where $Z$ is a $(-1)$-curve with  $ZG=0$, $ZN_0=1$ and $B_0Z=7$.
   \end{enumerate}
\end{enumerate}
\end{prop}

\proof We divide the proof into several steps.

\noindent\paragraph{Step~1:}We first show the existence of a $(-1)$-curve $G$ satisfying (2) and (3).\\

Note that $M_4=4K_W+D$. From \eqref{eq:Mj},
$$M_4^2=-2, K_WM_4=-6.$$
Since $M_3=3K_W+D$ is nef and big, the Kawamata-Viehweg vanishing theorem and the Riemann-Roch theorem yield
$$h^0(W,\mathcal{O}_W (M_4))=3\blue{.}$$
For any $\Theta \in |M_4|$, since $\Theta^2=-2$, there is an irreducible component $G$ of $\Theta$ such that
$\Theta G<0$. Then $G^2<0$ and $4K_WG<-DG\le 0$.
Thus $G$ is a $(-1)$-curve and $DG=3$. By \lemref{lem:-1curve2}, $GB_0=5$ and
$G$ intersects exactly one of the nodal curves. May assume $GN_0=1$.

\noindent\paragraph{Step~2:} We show the existence of the rational fibration $f \colon W \rightarrow \PP^1$
satisfying (1) and $|M_4|=2|F|+2G+N_0$.\\

Let $h\colon W \rightarrow W'$ be the contraction of $G+N_0$ by first blowing down $G$ and then blowing down the image of $N_0$.
Denote by $p'$ the point $h(G+N_0)$. Observe that $p'$ is not in any of the nodal curves $N_1', \ldots, N_8'$,
where $N_j'=h(N_j)$ for $j=1,\ldots, 8$.

The surface $W'$ is a smooth rational surface with eight nodal curves
$N_1', \ldots, N_8'$.
  Apply \cite[Theorem~3.2]{manynodes}, and then there is a rational fibration
 $f' \colon W' \rightarrow \PP^1$ with exactly singular fibres as follows:
 $$N_{2j-1}'+2G_j'+N_{2j}',\  j=1,2,3,4,$$
 where $G_j'$ is a $(-1)$-curve with $N_{2j-1}'G_j'=1$, $G_j'N_{2j}'=1$ for $j=1,\ldots, 4$.
Thus we have $f:=f'\circ h \colon W \rightarrow \PP^1$ is a rational fibration. Denote the general fibre of $f$ by $F$.
Clearly, $G, N_0, \ldots, N_8$ are contained in the singular fibres of $f$.

Now we calculate $DF$.
The determinant of the intersection number matrix of the following $13$ divisors $K_W, D, F, G, N_0, \ldots, N_8$
          is $2^{12}(DF-8)$. Since $\rho(W)=12$, the determinant is $0$.
So $DF=8$.
The direct computation shows that
$$M_3[M_4-(2F+2G+N_0)]=0,\  [M_4-(2F+2G+N_0)]^2=0.$$
Since $M_3$ is nef and big, $M_4\equiv 2F+2G+N_0$.
Since $h^0(W,\mathcal{O}_W (M_4))=3$ and $h^0(W,\mathcal{O}_W (2F))=3$, we conclude that $|M_4|=2|F|+2G+N_0$.

\noindent\paragraph{Step~3:}~Any irreducible component $\Gamma$ of $B_0$ is not contained in a fibre of $f$.\\

Note that since $D=2K_W+B_0$,
$$B_0\equiv -6K_W+2F+2G+N_0$$
by \eqref{eq:M4}. Also recall that $B_0$ is a disjoint union of smooth irreducible curves.
Then
$$\Gamma^2=\Gamma B_0=-6K_W\Gamma+2F\Gamma+2G\Gamma,$$
that is
$$2g(\Gamma)-2=-5K_W\Gamma+2F\Gamma+2G\Gamma.$$
Since $G^2=-1$ and $GB_0=5$, we have $B_0 \not \ge G$.
In particular, $G\Gamma \ge 0$.

Assume that $\Gamma$ is contained in some fibre of $f$. Then $F\Gamma=0$ and $\Gamma\cong \PP^1$.
So $$-2=-5K_W\Gamma+2G\Gamma.$$
Since $G$ is also contained in some fibre of $f$, $G\Gamma=1$ or $G\Gamma=0$.
Then $-5K_W\Gamma=-4$ or $-2$. This is impossible.

\noindent\paragraph{Step~4:} Now we consider the singular fibres of $f$ and complete the proof of the proposition.

Clearly $f$ contains the singular fibres
$$h^*(N_{2j-1}'+2G_j'+N_{2j}')=N_{2j-1}+2G_j+N_{2j}$$
where $G_j=h^*(G_j')$, $j=1,\ldots,4$.

We distinguish two cases:
\begin{itemize}
  \item Case~I:~$p'$ is contained in one of the singular fibres of $f'$, say $N_1'+2G_1'+N_2'$.
 \item Case~II:~$p'$ is not contained in the singular fibres of $f'$.
\end{itemize}
If Case~I occurs, then $p' \in G_1'$ since $G+N_0$ is disjoint from $N_1, \ldots, N_8$. Thus either $G_1$ is in the case (5)~(i) or $G_1=N_0+2G+Z$, where $Z$ is the strict transform of $G_1'$ by $h$. If the latter occurs, then $F\ge 4G$ and thus $8=DF\ge D(4G)=12$, a contradiction.

If Case~II occurs, then besides \eqref{eq:fibres}, the pullback by $h$ of the smooth fibre $F'$ of $f'$ containing $p'$
is either in the case (5)~(ii) or $N_0+2G+Z$ where $Z$  is the strict transform of $F'$ by $h$.

Suppose the later occurs. By Step 3 we consider the morphism $\gamma:=f|_{B_0} \colon B_0 \rightarrow \PP^1$.
Note that $\deg \gamma=FB_0=12$ and the ramification divisor $\Delta_\gamma$ of $\gamma$ satisfies
$$\Delta_\gamma \ge G|_{B_0}+\sum_{j=1}^4G_j|_{B_0}.$$
Then Riemann-Hurwitz formula yields
$$2p_a(B_0)-2\ge \deg\gamma\cdot(2g(\PP^1)-2)+GB_0+\sum_{j=1}^4G_jB_0$$
that is $4 \ge 5$, a contradiction.

We complete the proof of the proposition. \qed

\subsection{The case $\kappa(W)=-\infty$: the branch divisor $B_0$}
In this subsection, we continue to assume $\kappa(W)=-\infty$ and so $W$ is a smooth rational surface.
\begin{prop}\label{prop:branchdiv} The following are the possibilities of $B_0$:
\begin{enumerate}[\upshape (1)]
    \item $B_0=\substack{\Gamma_0 \\ (2, 0)}+\substack{\Gamma_1 \\ (2, 0)}+\substack{\Gamma_2 \\ (1,-2)}$;
     in this case, $\Gamma_0\equiv \Gamma_1\equiv -2K_W+F$ and $\Gamma_2 \equiv -2K_W+2G+N_0$.
    \item $B_0=\substack{\Gamma_0 \\ (3, 0)}+\substack{\Gamma_1 \\ (1, -2)}$; in this case,
     $\Gamma_0 \equiv -4K_W+2F$ and $ \Gamma_1\equiv -2K_W+2G+N_0$.
    \item $B_0=\substack{\Gamma_0 \\ (2, -2)}+\substack{\Gamma_1 \\ (2, 0)}$; in this case,
     $\Gamma_0 \equiv -4K_W+F+2G+N_0$ and $\Gamma_1\equiv -2K_W+F$.
    \item $B_0=\substack{\Gamma_0 \\ (3,-2)}$.
\end{enumerate}
\end{prop}
\begin{rem} There are examples for Proposition \ref{prop:branchdiv} (1) and (2), but no example for Proposition \ref{prop:branchdiv} (3) and (4). See Section \ref{exambran}.
\end{rem}
\begin{lem}\label{lem:gGamma}For any irreducible component $\Gamma$ of $B_0$,
\begin{enumerate}[\upshape (1)]
\item $F\Gamma$ is an even integer greater than or equal to $2$.
\item $g(\Gamma)\ge 1$.
\end{enumerate}
\end{lem}
\proof From Step~3 of the proof of \propref{prop:rationalfib},
$\Gamma$ is not contained in the fibres of $f$ for any irreducible component $\Gamma$ of $B_0$.

Set $d:=F\Gamma \ge 1$. Then $d=(N_3+2G_2+N_4)\Gamma=2G_2\Gamma$ and thus $d$ is an even integer with $d\ge 2$.
Let $\gamma:=f|_\Gamma \colon \Gamma \rightarrow \PP^1$.
The ramification divisor $\Delta_\gamma$ of $\gamma$ satisfies
$\Delta_\gamma\ge \sum_{j=1}^4 G_j|_\Gamma$.
Therefore $\deg \Delta_\gamma\ge \sum_{j=1}^4G_j\Gamma=2d$. The Riemann-Hurwitz formula shows that
$$2g(\Gamma)-2\ge (-2)d+2d=0.$$
Hence $g(\Gamma)\ge 1$\blue{.} \qed

\begin{proof}[Proof of \propref{prop:branchdiv}]
The fact $g(\Gamma)\ge 1$ in \lemref{lem:gGamma} excludes the cases (a), (c), (d) in \thmref{thm:known}~(2).

\paragraph{Exclusion of the case (f)~$B_0=\substack{\Gamma_0 \\ (3, 2)}+\substack{\Gamma_1 \\ (1, -4)}$}
\textrm{ }

Note that
$$(K_W+\Gamma_1)^2=2,~~(K_W+\Gamma_1)\Gamma_0=2,~~\Gamma_0^2=2.$$
The algebraic index theorem shows that
$$K_W+\Gamma_1\equiv \Gamma_0.$$
Therefore $F\Gamma_0=F\Gamma_1-2$.
Since $F\Gamma_0+F\Gamma_1=FB_0=12$, $F\Gamma_0=5$, a contradiction to \lemref{lem:gGamma}~(1).

\paragraph{Exclusion of the case (h)~$B_0=\substack{\Gamma_0 \\ (3, 2)}+\substack{\Gamma_1 \\ (1, -2)}+\substack{\Gamma_2 \\ (1,-2)}$}
\textrm{ }

Note that $\Gamma_0, \Gamma_1, \Gamma_2, N_0, \ldots, N_8$ is a basis of $\mathrm{Num}(W)\otimes_\ZZ \QQ$. Note that $K_W$ is orthogonal to $N_0, N_1, \ldots, N_8$ and
$K_W\Gamma_i=2$ for $i=0,1,2$. This shows that
$$K_W\equiv \Gamma_0-\Gamma_1-\Gamma_2.$$
Then $F\Gamma_0=F(\Gamma_1+\Gamma_2)-2$.
Since $F(\Gamma_0+\Gamma_1+\Gamma_2)=FB_0=12$, $F\Gamma_0=5$. This contradicts \lemref{lem:gGamma}~(1).\\

Now we give a linear equivalent expression of each component of $B_0$ by $K_W, F, G$ and $N_0$
for the remaining cases.
\paragraph{The case $B_0=\substack{\Gamma_0 \\ (2, 0)}+\substack{\Gamma_1 \\ (2, 0)}+
\substack{\Gamma_2 \\ (1,-2)}$}
\textrm{ }

Clearly, $\Gamma_0 \equiv \Gamma_1$.
We have
\begin{align*}
K_W\Gamma_1=2,~~K_W\Gamma_2=2.
\end{align*}
Also $F\Gamma_2=F(B_0-\Gamma_0-\Gamma_1)=12-2F\Gamma_1$. Set $x=F\Gamma_1$.
Then the matrix of intersection numbers of $K_W, F, \Gamma_1, \Gamma_2$ is
$$
\begin{pmatrix}
-2 & -2 & 2 & 2\\
-2 & 0  & x & 12-2x\\
2  & x  & 0 & 0\\
2  & 12-2x &0 & -2
\end{pmatrix}_{\textstyle \raisebox{2pt}{.}}
$$
Since $K_W, F, \Gamma_1, \Gamma_2$ are orthogonal to the nodal curves $N_0, N_1, \ldots, N_8$ and $\rho(W)=12$, the determinant of the matrix is $0$. That is
$$32x^2-272x+576=16(x-4)(2x-9)=0.$$
It follows that $x=4$, and so $F\Gamma_0=F\Gamma_1=F\Gamma_2=4$.

Hence we obtain $2K_W+\Gamma_1\equiv F$.

\paragraph{The case $B_0=\substack{\Gamma_0 \\ (3, 0)}+\substack{\Gamma_1 \\ (1,-2)}$}
\textrm{ }

We have
\begin{align*}
K_W\Gamma_0=4,~~K_W\Gamma_1=2.
\end{align*}
Also $F\Gamma_1=F(B_0-\Gamma_0)=12-F\Gamma_0$. Set $x=F\Gamma_0$.
Then the matrix of intersection numbers of $K_W, F, \Gamma_0, \Gamma_1$ is
$$
\begin{pmatrix}
-2 & -2 & 4 & 2\\
-2 & 0  & x & 12-x\\
4  & x  & 0 & 0\\
2  & 12-x &0 & -2
\end{pmatrix}_{\textstyle \raisebox{2pt}{.}}
$$
Since $K_W, F, \Gamma_0, \Gamma_1$ are orthogonal to the nodal curves $N_0, N_1, \ldots, N_8$ and $\rho(W)=12$, the determinant of the matrix is $0$. That is
$$32x^2-544x+2304=32(x-8)(x-9)=0.$$
It follows that $x=8$. Thus $F\Gamma_0=8$ and $F\Gamma_1=4$.

Hence we obtain $4K_W+\Gamma_0 \equiv 2F$.

\paragraph{The case $B_0=\substack{\Gamma_0 \\ (2, -2)}+\substack{\Gamma_1 \\ (2, 0)}$}
\textrm{ }

We have
\begin{align*}
K_W\Gamma_0=4,~~K_W\Gamma_1=2.
\end{align*}
Also $F\Gamma_1=F(B_0-\Gamma_0)=12-F\Gamma_0$. Set $x=F\Gamma_0$.
Then the matrix of intersection numbers of $K_W, F, \Gamma_0, \Gamma_1$ is
$$
\begin{pmatrix}
-2 & -2 & 4 & 2\\
-2 & 0  & x & 12-x\\
4  & x & -2 & 0\\
2  & 12-x &0 & 0
\end{pmatrix}_{\textstyle \raisebox{2pt}{.}}
$$
Since $K_W, F, \Gamma_0, \Gamma_1$ are orthogonal to the nodal curves $N_0, N_1, \ldots, N_8$ and $\rho(W)=12$, the determinant of the matrix is $0$. That is
$$32x^2-496x+1920=16(2x-15)(x-8)=0.$$
It follows that $x=8$. Thus $F\Gamma_0=8$ and $F\Gamma_1=4$.

Hence we obtain $2K_W+\Gamma_1 \equiv F$.

We complete the proof of the proposition.
\end{proof}

\section{The case $k=11$}
In this section, we consider the case $k=11$. We use the same notation in the introduction.
\subsection{Known results}
\begin{lem}[{\cite[Theorem~1.1]{bicanonical2}}]\label{lem:canonical}
The canonical divisor $K_S$ is ample.
\end{lem}

\begin{lem}[{\cite[Proposition~3.1 (i), (ii)]{involution}, \cite[Proposition 2.4 (i), (ii)]{involutioncampedelli},  \cite[Theorem~3.5~(ii)]{leeshin}}]
\label{lem:D}
Let $D: =2K_W+B_0$.
Then
    \begin{enumerate}[\upshape (1)]
        \item $\tpi^*D \equiv \e^*(2K_S)$, and $D$ is nef and big;
        \item $D^2=14, DK_W=0, K_W^2=-4, K_W B_0=8, B_0^2=-2$ and $\kappa(W)=-\infty$;
        \item If $DC=0$ for an irreducible curve $C$, then $C$ is one of the eleven
              nodal curves $N_0, \ldots, N_{10}$.
    \end{enumerate}
\end{lem}
\begin{lem}\label{lem:M_1}
Let $M_1:=K_W+D$.
Then
    \begin{enumerate}[\upshape (1)]
        \item $h^0(W, \O_W(M_1))=8$, $K_WM_1=-4$, $DM_1=14$ and $M_1^2=10$;
        \item $M_1$ is nef and big.
    \end{enumerate}
\end{lem}
\begin{proof}

         Since $D$ is nef and big, $D$ is $1$-connected.
         Thus $h^0(W, \O_W(M_1))=p_a(D)=8$.
         Moreover,
         $K_WM_1=K_W(D+K_W)=-4$, $DM_1=D(D+K_W)=14$ and $M_1^2=(K_W+D)M_1=10$.

         Assume that $M_1$ is not nef. Then there is an irreducible curve $C$ such that $M_1C <0$.
         Since $M_1$ is effective, we have that $C^2<0$.
         Also $K_WC=M_1C-DC\le M_1C <0$.
         Hence $C$ is a $(-1)$-curve and $DC=0$.
         This contradicts Lemma \ref{lem:D}~(3).
\end{proof}

\subsection{A Rational fibration}
\begin{prop}\label{prop:11rationalfib}After possibly renumbering the nodal curves $N_0, \ldots, N_{10}$, we have
\begin{align}
|2K_W+D|=3|F|+2G+N_0, \label{eq:11M2}
\end{align}
where $F$ is a general fibre of a rational fibration $f \colon W \rightarrow \PP^1$ and
\begin{enumerate}[\upshape (1)]
\item $DF=4, B_0F=8$;
\item $G$ is a $(-1)$-curve with $GN_0=1, DG=1$ and $B_0G=3$;
\item $G$ is disjoint from the nodal curves $N_1, \ldots, N_{10}$;
\item The fibration $f$ contains the following $5$ singular fibres:
         \begin{align}
         N_{2j-1}+2G_j+N_{2j},~~j=1, 2, 3, 4, 5, \label{eq:11fibres}
         \end{align}
         where for $j=1, 2, 3, 4, 5$, $G_j$ is a reduced curve with $G_j^2=-1$, $G_jN_{2j-1}=G_{j}N_{2j}=1$ and $G_jB_0=4$, and for $j=2, 3, 4, 5$,
         $G_j$ is a $(-1)$-curve.
\item There are two possibilities for the reduced curve $G_1$:
    \begin{enumerate}
     \item[(i)]$G_1=G+N_0+Z$, where $Z$ is a $(-2)$-curve with $ZG=0$, $ZN_0=ZN_1=ZN_2=1$ and $B_0Z=1$. In this case,
      \eqref{eq:11fibres} are exactly all the singular fibres of $f$.

      \item[(ii)]  $G_1$ is a $(-1)$-curve. In this case, besides \eqref{eq:11fibres}, $f$ contains another singular fibre
      $G+N_0+Z$, where $Z$ is a $(-1)$-curve with $ZG=0$, $ZN_0=1$ and $B_0Z=5$.
   \end{enumerate}
\end{enumerate}
\end{prop}

\proof We divide the proof into several steps.

\noindent\paragraph{Step~1:}We first show the existence of a $(-1)$-curve $G$ satisfying (2) and (3).\\

Note that $M_2:=K_W+M_1=2K_W+D$. From Lemmas \ref{lem:D} and \ref{lem:M_1},
$$M_2^2=-2, K_WM_2=-8.$$
Since $M_1$ is nef and big, the Kawamata-Viehweg vanishing theorem and the Riemann-Roch theorem yield
$$h^0(W,\mathcal{O}_W (M_2))=4.$$
For any $\Theta \in |M_2|$, since $\Theta^2=-2$, there is an irreducible component $G$ of $\Theta$ such that
$\Theta G<0$. Then $G^2<0$ and $2K_WG<-DG\le 0$.
Thus $G$ is a $(-1)$-curve and $DG=1$. So  $GB_0=G(D-2K_W)=3$.
By \eqref{eq:coveringdata}
$$2G\mL=G(B_0+N_0+\cdots +N_{10})=3+G(N_0+\cdots+N_{10}).$$
Therefore
$G(N_0+\cdots+N_{10})$ is an odd integer.\ So $G$ intersects one of the nodal curves, say $GN_0 \ge 1$.\ The similar argument in the proof of Lemma \ref{lem:-1curve2} gives that $GN_0=1$ and $G$ does not intersect $N_1+\cdots+N_{10}$.

\noindent\paragraph{Step~2:} We show the existence of the rational fibration $f \colon W \rightarrow \PP^1$
satisfying (1) and $|M_2|=3|F|+2G+N_0$.\\

Let $h\colon W \rightarrow W'$ be the contraction of $G+N_0$ by first blowing down $G$ and then blowing down the image of $N_0$.
Denote by $p'$ the point $h(G+N_0)$. Observe that $p'$ is not in any of the nodal curves $N_1', \ldots, N_{10}'$,
where $N_j'=h(N_j)$ for $j=1,\ldots, 10$.

The surface $W'$ is a smooth rational surface with ten nodal curves
$N_1', \ldots, N_{10}'$.
  Apply \cite{manynodes}, and then there is a rational fibration
 $f' \colon W' \rightarrow \PP^1$ with exactly singular fibres as follows:
 $$N_{2j-1}'+2G_j'+N_{2j}',\  j=1,2,3,4,5$$
 where $G_j'$ is a $(-1)$-curve with $N_{2j-1}'G_j'=1$, $G_j'N_{2j}'=1$ for $j=1,\ldots, 5$.
Thus we have $f:=f'\circ h \colon W \rightarrow \PP^1$ is a rational fibration.\ Denote the general fibre of $f$ by $F$.
Clearly, $G, N_0, \ldots, N_{10}$ are contained in the singular fibres of $f$.

Now we calculate $DF$.
The determinant of the intersection number matrix of the following $15$ divisors $K_W, D, F, G, N_0, \ldots, N_{10}$
          is $2^{11}(DF-4)(DF+8)$. Since $\rho(W)=14$, the determinant is $0$.
So $DF=4$.
The direct computation shows that
$$M_1[M_2-(3F+2G+N_0)]=0,\ [M_2-(3F+2G+N_0)]^2=0.$$
Since $M_1$ is nef and big, $M_2\equiv 3F+2G+N_0$.
Since $h^0(W,\mathcal{O}_W (M_2))=4$ and $h^0(W, \mathcal{O}_W (3F))=4$, we conclude that $|M_2|=3|F|+2G+N_0$.

\noindent\paragraph{Step~3:}~Any irreducible component $\Gamma$ of $B_0$ is not contained in a fibre of $f$.\\

Note that since $D=2K_W+B_0$,
$$B_0\equiv -4K_W+3F+2G+N_0$$
by \eqref{eq:11M2}. Also recall that $B_0$ is a disjoint union of smooth irreducible curves.
Then
$$\Gamma^2=\Gamma B_0=-4K_W\Gamma+3F\Gamma+2G\Gamma,$$
that is
$$2g(\Gamma)-2=-3K_W\Gamma+3F\Gamma+2G\Gamma.$$
Since $G^2=-1$ and $GB_0=3$, we have $B_0 \not \ge G$.
In particular, $G\Gamma \ge 0$.

Assume that $\Gamma$ is contained in some fibre of $f$. Then $F\Gamma=0$ and $\Gamma\cong \PP^1$.
So $$-2=-3K_W\Gamma+2G\Gamma.$$
Since $G$ is also contained in some fibre of $f$, $G\Gamma=1$ or $G\Gamma=0$.
Then $-3K_W\Gamma=-4$ or $-2$. This is impossible.

\noindent\paragraph{Step~4:} Now we consider the singular fibres of $f$ and complete the proof of the proposition.

Clearly $f$ contains the singular fibres
$$h^*(N_{2j-1}'+2G_j'+N_{2j}')=N_{2j-1}+2G_j+N_{2j}$$
where $G_j=h^*(G_j')$, $j=1,\ldots,5$.

We distinguish two cases:
\begin{itemize}
  \item Case~I:~$p'$ is contained in one of the singular fibres of $f'$, say $N_1'+2G_1'+N_2'$.
 \item Case~II:~$p'$ is not contained in the singular fibres of $f'$.
\end{itemize}
If Case~I occurs, then $p' \in G_1'$ since $G+N_0$ is disjoint from $N_1, \ldots, N_{10}$. Thus either $G_1$ is in the case (5)~(i) or $G_1=N_0+2G+Z$, where $Z$ is the strict transform of $G_1'$ by $h$ and it is a
$(-3)$-curve. If the latter occurs,  then $D \ge 4G$ and thus $4=DF\ge D(4G)=4$.  It follows that $DZ=0$ which is $B_0Z+2K_WZ=0$. So $B_0Z=-2$ and thus $B_0\ge Z$, a contradiction to Step~3.

If Case~II occurs, then besides \eqref{eq:11fibres}, the pullback by $h$ of the smooth fibre $F'$ of $f'$ containing $p'$
is either in the case (5)~(ii) or $N_0+2G+Z$ where $Z$  is the strict transform of $F'$ by $h$.

Suppose the later occurs. By Step 3 we consider the morphism $\gamma:=f|_{B_0} \colon B_0 \rightarrow \PP^1$.
Note that $\deg \gamma=FB_0=8$ and the ramification divisor $\Delta_\gamma$ of $\gamma$ satisfies
$$\Delta_\gamma \ge G|_{B_0}+\sum_{j=1}^{5}G_j|_{B_0}.$$
Then Riemann-Hurwitz formula yields
$$2p_a(B_0)-2\ge \deg\gamma\cdot (2g(\PP^1)-2)+GB_0+\sum_{j=1}^5G_jB_0$$
that is $6 \ge 7$, a contradiction.

We complete the proof of the proposition. \qed
\begin{rem}
The surface constructed in \cite{inoue} has an involution to support singular fibres of a rational fibration in Proposition \ref{prop:11rationalfib} (4) and (5) (i).
We do not have an example to support the case of Proposition \ref{prop:11rationalfib} (4) and (5) (ii). These remarks are mentioned by Mendes Lopes and Pardini in \cite[Remark 3.4 (i)]{bicanonical2}.
\end{rem}

\subsection{The branch divisor $B_0$}
\begin{prop}\label{prop:11branchdiv} The following are the possibilities of $B_0$:
\begin{enumerate}[\upshape (1)]
    \item $B_0=\substack{\Gamma_0 \\ (3, 0)}+\substack{\Gamma_1 \\ (2, -2)}$;
     in this case, $\Gamma_0\equiv -2K_W+2F$ and $\Gamma_1 \equiv -2K_W+F+2G+N_0$.
    \item $B_0=\substack{\Gamma_0 \\ (3, -2)}+\substack{\Gamma_1 \\ (2, 0)}$; in this case,
     $\Gamma_0 \equiv -3K_W+2F+2G+N_0$ and $ \Gamma_1\equiv -K_W+F$.
    \item $B_0=\substack{\Gamma_0 \\ (4,-2)}$.
\end{enumerate}
\end{prop}
\begin{rem}
There are examples for Proposition \ref{prop:11branchdiv} (1), but no example for Proposition \ref{prop:11branchdiv} (2) or (3). See Section \ref{exambran}.
\end{rem}
\begin{proof}
By \lemref{lem:D}~(2), $p_a(B_0)=4$.
Assume that $B_0$ is reducible. Let $B_0=\Gamma_0+\cdots+\Gamma_r$ be the irreducible decomposition with $r \ge 1$ and $g(\Gamma_0) \ge \ldots \ge g(\Gamma_r)$.
Then
\begin{align}
4=p_a(B_0)=g(\Gamma_0)+\cdots+g(\Gamma_r)-r\label{eq:g(Gamma)}.
\end{align}

Fix $i \in \{0, \ldots, r\}$.
From Step~3 of the proof of \propref{prop:11rationalfib},
$\Gamma_i$ is not contained in the fibres of $f$. Set $d_i:=F\Gamma_i \ge 1$. Then $d_i=(N_3+2G_2+N_4)\Gamma_i=2G_2\Gamma_i$. Thus $d_i$ is an even integer with $d_i\ge 2$.
Let $\gamma_i:=f|_{\Gamma_i} \colon \Gamma_i \rightarrow \PP^1$.
The ramification divisor $\Delta_{\gamma_i}$ of $\gamma_i$ satisfies
$\Delta_{\gamma_i}\ge \sum_{j=1}^5 G_j|_{\Gamma_i}$.
Therefore $\deg \Delta_{\gamma_i}\ge \sum_{j=1}^{5}G_j{\Gamma_i}=\frac52 d_i$. The Riemann-Hurwitz formula shows that
\begin{align}\label{eq:gGamma}
2g(\Gamma_i)-2\ge (-2)d_i+\frac52 d_i=\frac 12d_i,~i.e.~g(\Gamma_i)\ge \frac 14d_i+1.
\end{align}
Hence $g(\Gamma_i)\ge 2$.

By \eqref{eq:g(Gamma)}, $r=1$, $g(\Gamma_0)=3$ and $g(\Gamma_1)=2$ or $r=2$, $g(\Gamma_i)=2$ for $i=0,1,2$.

We now show that $G\Gamma_i=g(\Gamma_i)-1$ and $K_W\Gamma_i=F\Gamma_i=d_i$ for each $i=0, \ldots, r$.
Fix $i$. Recall from \eqref{eq:11M2} that $B_0\equiv -4K_W+3F+2G+N_0$.
Since $N_0\Gamma_i=0$ and $B_0\Gamma_i=\Gamma_i^2$, we have
$-4K_W\Gamma_i+3F\Gamma_i+2G\Gamma_i=\Gamma_i^2$.
By the adjunction formula, $K_W\Gamma_i+\Gamma_i^2=2g(\Gamma_i)-2$ and thus
$$3(F\Gamma_i-K_W\Gamma_i)=2(g(\Gamma_i)-1-G\Gamma_i).$$
In particular, $g(\Gamma_i)-1-G\Gamma_i$ is divisible by $3$.
Since $0 \le G\Gamma_i \le GB_0=3$ and $2 \le g(\Gamma_i) \le 3$,
it follows that $G\Gamma_i=g(\Gamma_i)-1$ and so $K_W\Gamma_i=F\Gamma_i=d_i$.

Assume $r=2$. Then we have seen $g(\Gamma_i)=2$ and $G\Gamma_i=1$ for $i=0, 1, 2$.
We may assume $F\Gamma_0 \ge F\Gamma_1 \ge F\Gamma_2$.
Since $F(\Gamma_0+\Gamma_1+\Gamma_2)=FB_0=8$ and $F\Gamma_i$ is a positive even integer for $i=0,1,2$,
we have $F\Gamma_0=4, F\Gamma_1=2$ and $F\Gamma_2=2$. So $K_W\Gamma_0=4, K_W\Gamma_1=2$ and $K_W\Gamma_2=2$.
Thus the adjunction formula yields that $\Gamma_0^2=-2, \Gamma_1^2=0$ and $\Gamma_2^2=0$.
A direct computation shows that
$$M_1(\Gamma_1+K_W-F)=0, (\Gamma_1+K_W-F)^2=0.$$
Since $M_1$ is nef and big, the algebraic index theorem shows that $\Gamma_1\equiv -K_W+F$.
Similarly, $\Gamma_2\equiv -K_W+F$. So $\Gamma_1 \equiv \Gamma_2$, and then $|\Gamma_1|$ is a base point free pencil. Thus $\Gamma_0$ is contained in some member of $|\Gamma_1|$ because of $\Gamma_0\Gamma_1=0$.
Since $F$ is nef, $4=F\Gamma_0 \le F\Gamma_1=2$, a contradiction.

Hence $r=1$.
And then $g(\Gamma_0)=3, g(\Gamma_1)=2$, $G\Gamma_0=2$ and $G\Gamma_1=1$.
Also $F(\Gamma_0+\Gamma_1)=FB_0=8$.
By \eqref{eq:gGamma}, $F\Gamma_1=d_1\le 4$.

If $F\Gamma_1=2$, then $K_W\Gamma_1=F\Gamma_1=2$. The adjunction formula gives $\Gamma_1^2=0$.
As above, we have $\Gamma_1 \equiv -K_W+F$.
And so $\Gamma_0=B_0-\Gamma_1 \equiv -3K_W+2F+2G+N_0$.
Note that $\Gamma_0^2=-2$. So $\Gamma_0$ and $\Gamma_1$ are in the case (2).

If $F\Gamma_1=4$, then $K_W\Gamma_1=F\Gamma_1=4$. The adjunction formula gives $\Gamma_1^2=-2$.
As above we conclude that  $\Gamma_1 \equiv -2K_W+F+2G+N_{0}$. And so $\Gamma_0 =B_0-\Gamma_1 \equiv -2K_W+2F$. Note $\Gamma_0^2=0$.
So $\Gamma_0$ and $\Gamma_1$ are in the case (1).

We complete the proof of the proposition.
\end{proof}

\section{Examples and questions}\label{exambran}
To the best knowledge of the authors, every known example of smooth minimal surfaces
of general type with $p_g=0$ and $K^2=7$ admits at least one involution.

The first examples are constructed by Inoue \cite{inoue} and they are called by Inoue surfaces.
These surfaces are $\ZZ/2\ZZ\times\ZZ/2\ZZ$-covers of the four nodal cubic surface \cite{bicanonical1} and they form a $4$-dimensional family. They are studied in \cite{inouemfd, chenshin18, leeshin}. In particular, the quotient surfaces and the branch divisors of the involutions are calculated in \cite[Page~137]{leeshin}:
\begin{table}[h]
\centering

\begin{tabular}{|c|c|c|c|c|c|}
\hline
                   & $k$      &  $K_{W_{i}}^{2}$ & $B_{0}$ & $W_{i}$ & Case \\
\hline
$(S,\gamma_{1})$  &  $11$     &  $-4$            & {\scriptsize $\substack{\Gamma_{0}\\(3,0)}+\substack{\Gamma_{1}\\(2,-2)}$} & rational & (7)(a)\\
\hline
$(S,\gamma_{2})$  &  $9$      &  $-2$            & {\scriptsize $\substack{\Gamma_{0}\\(3,0)}+\substack{\Gamma_{1}\\(1,-2)}$} & birational to an Enriques surface & (5)(b)\\
\hline
$(S,\gamma_{3})$  &  $9$      &  $-2$            &
{\scriptsize $\substack{\Gamma_{0}\\(2,0)}+\substack{\Gamma_{1}\\(2,0)}+\substack{\Gamma_{2}\\(1,-2)}$} & rational & (6)(a)\\
\hline
\end{tabular}
\end{table}

The next examples are constructed by the first named author as $\ZZ/2\ZZ\times\ZZ/2\ZZ$-covers of the six nodal del Pezzo surfaces of degree one \cite{chennew}. It is a three dimensional family. The quotient surfaces and the branch divisors of the involutions are calculated in
\cite[Proposition~5.1 and Remark~5.1]{chennew}:

\begin{table}[h]
\centering
\begin{tabular}{|c|c|c|c|c|c|}
\hline
                   & $k$      &  $K_{W_{i}}^{2}$ & $B_{0}$ & $W_{i}$ & Case \\
\hline
$(S, g_{1})$  &  $9$     &  $-2$            & {\scriptsize $\substack{\Gamma_{0}\\(3,0)}+\substack{\Gamma_{1}\\(1,-2)}$} & rational & (6)(b)\\
\hline
$(S, g_{2})$  &  $9$      &  $-2$            & {\scriptsize $\substack{\Gamma_{0}\\(3,-2)}$ } & birational to an Enriques surface & (5)(a)\\
\hline
$(S, g_{3})$  &  $7$      &  $0$            & {\scriptsize $\substack{\Gamma_{1}\\(2,-2)}$} & smooth minimal $\kappa=1$ & (4)(a)\\
\hline
\end{tabular}
\end{table}

The first and second named authors \cite{chenshin21} constructed a two dimensional family of surfaces described as $\ZZ/2\ZZ\times\ZZ/2\ZZ$-covers of certain rational surfaces with Picard number three, with eight nodes and with two elliptic fibrations. It turns out this family contains the surface constructed by Rito \cite{rito15}. The quotient surfaces and the branch divisors of the involutions are calculated in \cite[Proposition~5.1 and Remark~5.3]{chenshin21}:
\begin{table}[h]
\centering
\begin{tabular}{|c|c|c|c|c|c|}
\hline
                   & $k$      &  $K_{W_{i}}^{2}$ & $B_{0}$ & $W_{i}$ & Case\\
\hline
$(S, g_{1})$  &  $9$     &  $-2$            & {\scriptsize $\substack{\Gamma_{0}\\(3,-2)}$} & birational to an Enriques surface & (5)(a)\\
\hline
$(S, g_{2})$  &  $7$      &  $0$            & {\scriptsize $\substack{\Gamma_{0}\\(2,-2)}$} & smooth minimal $\kappa=1$ & (4)(a)\\
\hline
$(S, g_{3})$  &  $5$    &  $2$            &  {\scriptsize$\substack{\Gamma_{1}\\(1,-2)}$ }& smooth minimal $\kappa=2$ & (1) \\
\hline
\end{tabular}
\end{table}

Moreover Calabri and Stagnaro \cite{calabristagnaro} construct minimal surfaces of general type with $p_g=0$ and $K^2=7$ by $\ZZ/2\ZZ$-covers of a certain rational surface. Their surfaces provide $k=9$, $K_W^2=-2$ and a rational surface $W$. They use computer programs to show the existences of the branch divisors of the $\ZZ/2\ZZ$-covers. From the contruction, it is easy to see these examples are in the case 6~(b) in \thmref{thm:main}: $\substack{\Gamma_{0}\\(3,0)}+\substack{\Gamma_{1}\\(1,-2)}$.
Compared to the surfaces constructed by \cite{chennew} (see the first row of the second Table above), it naturally gives rise to the following question:
\begin{theorem*}\label{question1}
Are the surfaces constructed in \cite{calabristagnaro} in the same family of surfaces in  \cite{chennew}?
\end{theorem*}

From all these examples, the following table \ref{table:k,W,branch,existence}
shows those possibilities and existences in \thmref{thm:main}.

\begin{table}[h]
\begin{threeparttable}
\centering
\begin{adjustbox}{width=1\textwidth}
	\begin{tabular}{|c|c|l|l|c|}
	\hline
	{\small $k$} & {\small$K_{W}^{2}$}  &  {\small $W$}  &  {\small $B_{0}$} & {\small existence}\\
	\hline
	{\small 5} & {\small 2}  &   {\small smooth minimal of general type} & {\scriptsize   $\substack{\Gamma_{0}\\(1,-2)}$} & {\small \cite{chenshin21, rito15}} \\
    	\hline {\small 7}  & {\small 1}  &  {\small smooth minimal of general type}  &  {\scriptsize  $\substack{\Gamma_{0}\\(3,2)}$}  & {\small unknown}\\
	\hline {\small 7} &  {\small 0}  &  {\small of general type whose smooth} & {\scriptsize $\substack{\Gamma_{0}\\(2,-2)}$} & {\small unknown} \\ \cline{4-5}
 	           &       &   {\small  minimal model has $K^2=1$} & {\scriptsize $\substack{\Gamma_{0}\\(2,0)}+\substack{\Gamma_{1}\\(1,-2)}$} &  {\small unknown} \\ \cline{3-5}
 	           &       &   {\small smooth minimal properly elliptic} & {\scriptsize  $\substack{\Gamma_{0}\\(2,-2)}$} &  {\small \cite{chennew}, \cite{chenshin21, rito15}}  \\  \cline{4-5}
  	          &       &   & {\scriptsize  $\substack{\Gamma_{0}\\(2,0)}+\substack{\Gamma_{1}\\(1,-2)}$} &
 {\small unknown}\\
	\hline {\small 9} & {\small $-2$}   & {\small birational to an Enriques surface}   & {\scriptsize $\substack{\Gamma_{0}\\(3,-2)}$}  & {\small  \cite{chennew},  \cite{chenshin21, rito15}}  \\ \cline{4-5}
	             &          &                                                       &  {\scriptsize $\substack{\Gamma_{0}\\(3,0)}+\substack{\Gamma_{1}\\(1,-2)}$} & 	{\small \cite{inouemfd, inoue, leeshin, bicanonical1}} \\ \cline{3-5}
	            &          &  {\small rational }                                        &  {\scriptsize $\substack{\Gamma_{0}\\(2,0)}+\substack{\Gamma_{1}\\(2,0)}+\substack{\Gamma_{2}\\(1,-2)}$} & {\small \cite{inouemfd, inoue, leeshin, bicanonical1}\tnote{a}} \\  \cline{4-5}
 	           &          &                                                      &  {\scriptsize $\substack{\Gamma_{0}\\(3,0)}+\substack{\Gamma_{1}\\(1,-2)}$} &  {\small \cite{calabristagnaro}, \cite{chennew}} \\ \cline{4-5}
 	          &          &                                                      &  {\scriptsize $\substack{\Gamma_{0}\\(2,-2)}+\substack{\Gamma_{1}\\(2,0)}$} & {\small unknown} \\ \cline{4-5}
 	          &          &                                                      &  {\scriptsize $\substack{\Gamma_{0}\\(3,-2)}$} & {\small unknown}\\

	\hline
	{\small 11} & {\small $-4$} & {\small rational} & {\scriptsize $\substack{\Gamma_{0}\\(3,0)}+\substack{\Gamma_{1}\\(2,-2)}$} & {\small  \cite{inouemfd, chenshin18, inoue, leeshin, bicanonical1}\tnote{b}}  \\ \cline{4-5}
	    &         &              & {\scriptsize $\substack{\Gamma_{0}\\(3,-2)}+\substack{\Gamma_{1}\\(2,0)}$} & {\small unknown}
\\ \cline{4-5}
 	  &          &                                                      &  {\scriptsize $\substack{\Gamma_{0}\\(4,-2)}$} & {\small unknown}\\
	\hline
	\end{tabular}
\end{adjustbox}
\caption{Existence of branch divisors}
 \label{table:k,W,branch,existence}
 	\begin{tablenotes}
            \item[a] This case is classified by \thmref{thm:threecpnts}.
            \item[b] This case is classified by \cite[Theorem 1.1]{chennew}.
        \end{tablenotes}
     \end{threeparttable}

\end{table}

We also raise the following question:

\begin{theorem*}\label{question2}
The existence of examples for cases (2), (3)(a), (3)(b), (4)(b), (6)(c), (6)(d), (7)(b) and (7)(c)
in \thmref{thm:main}.
\end{theorem*}

On the other hand, only the case (7)(a) in \thmref{thm:main} is completely classified (see \cite{chenshin18}). Now we classify the case (6)(a) by proving \thmref{thm:threecpnts}.

\begin{proof}[Proof of \thmref{thm:threecpnts}]
By \thmref{thm:main}, we see that $k=9$ and $W$ is a rational surface.
Then by \propref{prop:branchdiv}~(1), $\Gamma_0\equiv \Gamma_1$. Since $\Gamma_0\cap \Gamma_1=\emptyset$,
it follows that $|\Gamma_0|$ yields a genus $2$ fibration on $W$. Since $\Gamma_0N_0=\ldots=\Gamma_0N_8=0$,
the nodal curves $N_0, \ldots, N_8$ are contained in the fibres.
So this fibration induces a fibration of genus $2$ on $\Sigma$.
Denote it by $\iota \colon \Sigma \rightarrow \mathbb{P}^1$.

Observe that a general fibre $I$ of $\iota$ is disjoint from the branch locus of $\pi$.
Therefore $\pi^*I \rightarrow I$ is an \'etale cover of degree $2$.

If $\pi^*I$ is a disjoint of two smooth irreducible curves of genus $2$, then by the Stein factorization of $\iota\circ \pi$,
$S$ admits a fibration of genus $2$. This contradicts \cite[Theorem~2]{xiao}.
We conclude that $\pi^*I$ is a smooth irreducible curve of genus $3$.
It is well known that if a genus $3$ curve is an \'etale cover of a genus $2$ curve, then it is hyperelliptic.

Therefore $\iota':=\iota\circ \pi$ is a hyperelliptic fibration of genus $3$ on $S$.
It induces an involution $\tau$ on $S$ and it is clear that $\tau \not=\sigma$.
By \cite[Theorem~1.2]{chenauto}, $\langle \tau, \sigma \rangle \cong \ZZ/2\ZZ \times \ZZ/2\ZZ$.
Recall that $R$ is the divisorial part of the fixed locus of $\sigma$ and denote by $R_\tau$ the one of $\tau$.

Denote by $I'$ a general fibre of $\iota'$.
Then $I'R_\tau=8$ since $\tau$ is the hyperelliptic involution on $I'$.
Note that $2\Lambda_i=\pi^*\eta(\Gamma_i)\equiv\pi^*I=I'$ for $i=0, 1$, where $\Lambda_i=\pi^{-1}(\eta(\Gamma_i))$ for
$i=0,1,2$. Note that $R=\Lambda_0+\Lambda_1+\Lambda_2$.
So $R_\tau R =R_\tau (\Lambda_0+\Lambda_1+\Lambda_2) =8+R_\tau \Lambda_2 \ge 8$.
Therefore by \cite[Theorem~1.1]{commuting}, $R_\tau R=9$ and the pair $(S, \langle \tau, \sigma \rangle)$ must be the case (a) of \cite[Theorem~1.1]{commuting}.
\end{proof}

We expect that the cases where $B_0$ has two irreducible components should be completely classified.

 \section*{Acknowledgments}
The second named author was supported by Basic Science Research Program through the National Research Foundation of Korea(NRF) funded by the Ministry of Education(No. RS-2023-00241086).

\medskip
\noindent Yifan Chen,\\
School of Mathematical Sciences, Beihang University, 9 Nanshan Street,
Shahe Higher Education Park, Changping, Beijing, 102206, P. R. China\\
Email:~chenyifan1984@buaa.edu.cn\\\smallskip

\noindent YongJoo Shin,\\
Department of Mathematics, Chungnam National University,
Science Building 1, 99 Daehak-ro, Yuseong-gu, Daejeon 34134, Republic of Korea\\
Email:~haushin@cnu.ac.kr\\\smallskip

\noindent Han Zhang,\\
School of Mathematical Sciences, Beihang University, 9 Nanshan Street,
Shahe Higher Education Park, Changping, Beijing, 102206, P. R. China\\
Email:~zh20010209@buaa.edu.cn
\end{document}